\documentclass[12pt]{article}
\usepackage{fullpage}
\usepackage{amsmath}
\usepackage{amssymb}
\usepackage{amsthm}
\usepackage[final]{graphicx}
\usepackage{mdwlist}
\usepackage{cite}
\usepackage{subfigure}

\usepackage{marginnote}
\usepackage{color}

\newtheorem{theorem}{Theorem}[section]
\newtheorem{lemma}[theorem]{Lemma}

\newtheorem{corollary}[theorem]{Corollary}

\theoremstyle{definition}
\newtheorem{definition}[theorem]{Definition}

\theoremstyle{remark}
\newtheorem{example}[theorem]{Example}
\newtheorem{remark}[theorem]{Remark}

\begin{document}

\title{A Game Theoretic Perspective on Network Topologies}
\author{Shaun Lichter, Christopher Griffin, and Terry Friesz}
\date{April 4, 2011}

\maketitle
\begin{abstract}
We extend the results of Goyal and Joshi (S. Goyal and S. Joshi. Networks of collaboration in oligopoly. Games and Economic behavior, 43(1):57-85, 2003), who first considered the problem of collaboration networks of oligopolies and showed that under certain linear assumptions network collaboration produced a stable complete graph through selfish competition. We show with nonlinear cost functions and player payoff alteration that stable collaboration graphs with an arbitrary degree sequence can result. As a by product, we prove a general result on the formation of graphs with arbitrary degree sequences as the result of selfish competition. Simple motivating examples are provided and we discuss a potential relation to Network Science in our conclusions.
\end{abstract}

\section{Introduction}
\label{sec: network formation}
In this paper we model the emergence of collaborations among players (e.g., firms) as a strategic network formation game \cite{dutta1997}, by allowing selfish agents to choose with which other agents they would like to form a link.  Each agent has the option to deny a link to another agent, so the formation of a link requires the cooperation of both players.  A value function assigns a value to each particular graph and this value is distributed to agents by an allocation function (or allocation rule).  This distribution of value drives a player's preference for particular graph structures. The primary objective of this paper is to  extend and generalize a model by Goyal and Joshi \cite{goyal2003} for collaboration of oligopolistic firms to provide a model that, under different conditions, admits a stable collaboration graph with symmetric or asymmetric degree distributions. The primary contribution of this work is showing that nonlinear pricing and collaboration incentives exist to create arbitrary stable network structures. This extends the work of Goyal and Joshi who showed that linear pricing schemes lead to stable graphs that are complete only. Thus the occurrence of exotic graphs (e.g., power law graphs) in natural collaborations may simply be the result of nonlinear incentives.

In particular circumstances a modeler may wish to predict how a network will change if it is perturbed in some way.  Specifically, a network may be changed by a player entering or exiting the game or when a link is added or deleted. For example, Verizon might be interested in knowing if AT\&T and T-Mobile merged, whether the new network would be stable and how the merger might affect the market.  Models that take a statistical in nature provide forecasts for the behavior of large networks, since bulk measures are used, but have difficulty predicting the stability of smaller network of interest.  Unfortunately, the models developed in this paper also have some limitations in that they require some knowledge of the objective functions of players. Further, the models developed in this paper are static models which provide insight into the stability of particular networks, but predictive capabilities may require a dynamic model of objective functions.  These two limitations are discussed further in Section \ref{sec: Future Directions} and are the driving motivation of continued work in this area.  Nonetheless, the models developed in this paper provide a foundation to develop modeling approaches that can predict the reaction of networks to specific changes that may take place.

The research on network formation games is extensive. We summarize a few key results germane to our presentation. Myerson \cite{myerson1977} considers cooperation structures, which we call graphs, but he considers only whether agents are connected, not the structure of the graph of connections.  Nonetheless, Myerson was the first to consider the strategic formation of networks.

Jackson and Wolinsky model social relationships where each player receives benefits at a cost from their friends and a lower benefit from friends of friends, but with no cost.  Jackson and Wolinsky investigate network formation via the stability of particular graph structures \cite{jackson1996}.  An efficient graph structure is one with the maximum value, that is, the total value of all agents is maximized (see Section \ref{sec: Notation} for a mathematical definition).  Jackson and Wolinksy found that efficient graph structures are not always stable, but their analysis used \textit{pairwise stability} (see Section \ref{sec: Notation}), which restricts the capabilities of agents and assumes a symmetric allocation rule.

Dutta and Mutuswami construct an allocation rule that ensures the existence of a strongly stable efficient graph where the allocation rule is symmetric on this subset of strongly stable graphs \cite{dutta1997}. Bala and Goyal \cite{bala2000} consider a model where links allow access to a noncompetitive product such as information. The cost of a link is attributed to the agent initiating the formation of the link.

Furusawa and Konishi \cite{furusawa2007} model free trade networks in which a link between countries represents a free trade agreement and the absence of a link results in a tariff.  The supply and demand functions then induce a value function and allocation rules over the countries.  Belleflamme and Bloch \cite{belleflamme2004} model market sharing between firms where a link  between two firms represents an agreement not to infringe on each other's market and the absence of a link implies the firms will compete in the same market. Again, supply and demand induce a value function and an allocation rule over the firms. Labor markets \cite{calvo2007} and co-author networks \cite{jackson1996} have also been modeled using game theoretic networks.

Networks allow a model to portray intricate relationships among players that cannot otherwise be described.  For example, non-transitive relationships may be modeled as in the case when (e.g.), Player 1 is connected to Players 2 and 3, but Players 2 and 3 are not connected to one another.  This generality leads to another important aspect of a network model: how do relationships of one pair of players affect a third player? In some cases, a positive benefit may be reaped by an outside party; other times, a negative affect may be imposed on an outside party.  Networks allow relationships to be modeled as an externality for outside players, but the nature of the externality is dependent on the application and the model.  Currarini \cite{currarini2002} investigates the nature of externalities in networks by modeling the value of a graph as a function of the components of the graph.  Particularly, any graph is partitioned into one or more components and the value may be a function of the partition.

The remainder of this paper is organized as follows: In Section \ref{sec: Notation} we provide notation used throughout the remainder of this paper. In Section \ref{sec: Beyond Oligopolies} we state and prove a general result on game theoretic constructions of graphs with arbitrary degree sequences. In Section \ref{sec:AppOlig} we apply this result to extend the work of Goyal and Joshi \cite{goyal2003} and provide the main results of the paper. In Section \ref{sec:Conclusion} we provide conclusions and suggest a bridge between this work and recent work in Network Science. Finally in Section \ref{sec: Future Directions} we discuss future directions of this research as well as new research directions that would make this work more applicable to the analysis of real-world organizations.


\section{Notational Preliminaries} \label{sec: Notation}
We follow the notational conventions common in the graph formation literature \cite{jackson1996,dutta1997,jackson2003}.  Let $N=\{1,2, \ldots n\}$ be the set of nodes in a graph, which will represent the players. In the case of oligopolies, a firm is considered a player. The set of links in the graph is a set of pairs of nodes (subsets of $N$ of size two).  A graph $g$ is a set of links (set of subsets of $N$ of size two).  The graph $g^{c}$ is the complete graph (all subsets of $N$ of size two) and $G$ is the set of all graphs over the node set $N$, that is, $G=\{g : g \subseteq g^{c} \}$.  For a graph $g \in G$, the graph $g+ij$ is the graph in which link $ij$ is added to $g$ while $g-ij$ is the graph in which link $ij$ is removed from graph $g$. Let $\eta_{i}(g):G \rightarrow \mathbb{R}$ denote the degree of node $i$ in graph $g$ and $\eta(g):G \rightarrow \mathbb{R}^n$ be the degree sequence of the graph $g$; i.e., $\eta(g)=(\eta_{1}(g),\eta_{2}(g),\ldots, \eta_{n}(g))$.  Let $[g]_{\eta}$ be the equivalence class of graphs with the same degree sequence as $g$.  A degree sequence $\mathbf{d}=\{d_1,\dots,d_n\}$ on $n$ nodes is \textit{graphical}, if there exists a graph $g$ with $n$ nodes that has degree sequence $\mathbf{d}$.  Equivalently, a degree sequence $\mathbf{d}$ is graphical if $\eta^{-1}(\mathbf{d}) \neq \emptyset$.

The value of a graph $g$ is the total value produced by agents in the graph; we denote the value of a graph as the function $v:G \rightarrow \mathbb{R}$ and the set of all such value functions as $V$. A graph $g$ is strongly efficient if $v(g) \geq v(g') \; \; \forall \; g' \in G$.  An allocation rule $Y:V \times G \rightarrow \mathbb{R}^{N}$ distributes the value $v(g)$ among the agents in $g$.  Denote the value allocated to agent $i$ as $Y_{i}(v,g)$.  Since, the allocation rule must distribute the value of the network to all players, it must be \textit{balanced}; i.e., $\sum_{i} Y_{i}(v,g)=v(g)$ for all $(v,g) \in V \times G$.  The allocation rule governs how the value of the graph is distributed and thus plays a significant role in the model.  The allocation rule may be used to model a free market, utilities of players, or system rules such as taxation or the provision of subsidies to particular players.  Throughout this paper, we denote the game $\mathcal{G}= \mathcal{G}(v,Y,N)$ as the game played with value function $v$ and allocation rule $Y$ over nodes $N$.

The strategy set of Player $i$, $S_{i}$, is the set of nodes to which Player $i$ may request to link and $S=\prod_{i \in N}{S_{i}}$ is the strategy space for all players.  The graph $g$ induced by strategy $s \in S$ is $g(s)=\{ij : j \in s_{i}, i \in s_{j}\}$. Here $s_i \in S_i$ is the strategy for Player $i$ in the overall strategy $s$.

The payoff of Player $i$ is defined as $f^{\mathcal{G}}_{i}(s)=Y_{i}(v,g(s))$.  This is interpreted as the payoff that player $i$ receives from the strategy $s$ in game $\mathcal{G}$ and is the proportion of the value $v(g(s))$ distributed by $Y$.  This model admits a Nash equilibrium in which each player chooses to link with no other players and no network forms.  Various authors argue that this is an inadequacy and therefore it is necessary to make refinements to the notion of equilibrium \cite{jackson2003,dutta1997}.

Jackson and Wolinksy use \textit{pairwise stability} to model stable networks without the use of non-cooperative Nash equilibrium \cite{jackson1996}.
\begin{definition}
A network $g$ with value function $v$ and allocation rule $Y$ is pairwise \textit{stable} if (and only if):
\begin{enumerate*}
\item for all $ij \in g$, $Y_{i}(v,g) \geq Y_{i}(v,g - ij)$ and
\item for all $ij \not\in g$, if $Y_{i}(v,g+ij) > Y_{i}(v,g)$, then $Y_{j}(v,g+ij) < Y_{j}(v,g)$
\end{enumerate*}
\end{definition}
Pairwise stability implies that in a stable network, for each link that exists, (1) both players must benefit from it and (2) if a link can provide benefit to both players, then it in fact must exist.  Jackson notes that pairwise stability may be too weak because it does not allow groups of players to add or delete links, only pairs of players \cite{jackson2003}.  Deletion of multiple links simultaneously has been considered in \cite{belleflamme2004}.

Goyal and Joshi \cite{goyal2006} provide a model that allows a player's allocation to be dependent on the number of links the player has.  That is:
\begin{equation}
Y_{i}(g)=\pi_{i}(g)-\gamma\eta_{i}(g)
\end{equation}
Here $\pi_{i}$ is a profit function, $\eta_{i}(g)$ is the number of links player $i$ has in $g$, and $\gamma$ is a link maintenance cost.  However, $\gamma$ may be negative to model a benefit of having more links.  The value function is then defined implicitly so that $Y_{i}(g)$ is a balanced allocation rule.

\section{A General Result on Graph Stability}
\label{sec: Beyond Oligopolies}
In this section, we show how to construct an allocation rule $Y_{i}(g)$ to ensure that a graph with a given degree sequence is stable.  To do this, we present a set of properties for the objectives of players that result in a stable graph with an arbitrary degree sequence.  Recall, a stable network is one in which no player has an incentive to drop a link or request an additional link to a player who reciprocates such an incentive.  Consider the scenario where for $i = 1,\dots,n$, Player $i$ has a desired degree $k_i$.  For example, some scientific authors may like to collaborate with only a few individuals, while some like to collaborate with many others.  Ultimately, when a collaboration network is formed, each player $i$ receives the number of links given by $\eta_{i}(g)$. In the game where each player's objective penalizes them for incurring a degree $\eta_{i}(g)$ not equal to their desired degree $k_i$ (under some conditions), the graph with a degree distribution of $k=(k_1,k_2,\ldots,k_n)$ will be stable.

With a slight abuse of the notation from Section \ref{sec: Notation}, we first define an allocation (payoff) to each player $i$ as $Y_{i}:G \rightarrow \mathbb{R}$ which is consistent with an allocation function  $Y:G \times V \rightarrow \mathbb{R}$ and an induced value function $V:G \rightarrow \mathbb{R}$ with $v(g)=\sum_{i}{Y_{i}(g)}$.

\begin{theorem}
Let $\mathbf{d} = (k_1,\dots,k_n)$ be a desired degree sequence for $n$ players in the node set $N$. Assume that Player $i$ wishes to maximize
objective $Y_{i}(g)=-f_{i}(\eta_{i}(g))=-f(\eta_{i}(g)-k_{i})$ (or minimize $f_{i}(\eta_{i}(g))=f(\eta_{i}(g)-k_{i})$), where $f:\mathbb{R} \rightarrow \mathbb{R}$ is a convex function with minimum at $0$.  Let $v$ be the balanced value function induced from the allocation rule $Y = (Y_1,\dots,Y_n)$. If $\eta^{-1}(\mathbf{d})$ is non-empty (i.e., $\mathbf{d}$ is graphic) then any graph $g$ such that $\eta(g) = \mathbf{d}$ is pairwise stable for the the game $\mathcal{G}(v,Y,N)$.
\label{thm: general arbitrary dist}
\end{theorem}

\begin{proof}
Let $g \in \eta^{-1}(\mathbf{d})$ be a graph so that $\eta_{i}(g)=k_{i}$ and let $g-ij$ be the graph $g$ after the link $ij$ is dropped from $g$ and define $g+ij$ to be the graph $g$ after the link $ij$ is added to graph $g$.  Observe that:
\begin{gather*}
\eta_{i}(g-ij)=k_{i}-1,\quad
\eta_{j}(g-ij)=k_{j}-1,\quad
\eta_{l}(g-ij)=k_{l}, \\
\eta_{i}(g+ij)=k_{i}+1,\quad
\eta_{j}(g+ij)=k_{j}+1,\quad
\eta_{l}(g+ij)=k_{l}
\end{gather*}
where $l \not\in \{i,j\}$

Since $f$ is convex with minimum at $0$ and each $f_{i}$ are simply shifts of $f$ so that they are minimized at $k_{i}$, we know that $f_{i}(k_{i})$ is a minimum of the function $f_i$.  Thus $Y_{i}(g)$ is maximized when  $g \in \eta^{-1}(\mathbf{d})$. It now suffices to confirm pairwise stability:
\begin{gather*}
Y_{i} (\eta_{i}(g-ij))-Y_{i} (\eta_{i}(g)) =Y_{i}(k_{i}-1)-Y_{i}(k_{i})<0\\
Y_{i} (\eta_{i}(g+ij))-Y_{i} (\eta_{i}(g)) =Y_{i}(k_{i}+1)-Y_{i}(k_{i})<0
\end{gather*}
Thus graph $g \in \eta(\mathbf{d})$ is stable because there is no firm $i$ willing to drop an arbitrary existing link $ij$ or add an arbitrary missing link $ij$.
\end{proof}

\begin{remark}
From Theorem \ref{thm: general arbitrary dist} it is interesting to note that the stability of the class $\eta^{-1}(\mathbf{d})$ is not unique. We can illustrate this fact using the following counterexample. Let $\mathbf{d} = (1,1,1,2,3)$ with $N = \{1,\dots,5\}$ and let $f(x) = x^2$, a simple convex function with minimum at zero, from the theorem. The following two networks are both pairwise stable solutions to the game $\mathcal{G} = (v,Y,N)$:
\begin{figure}[htbp]
\centering
\includegraphics[scale=0.65]{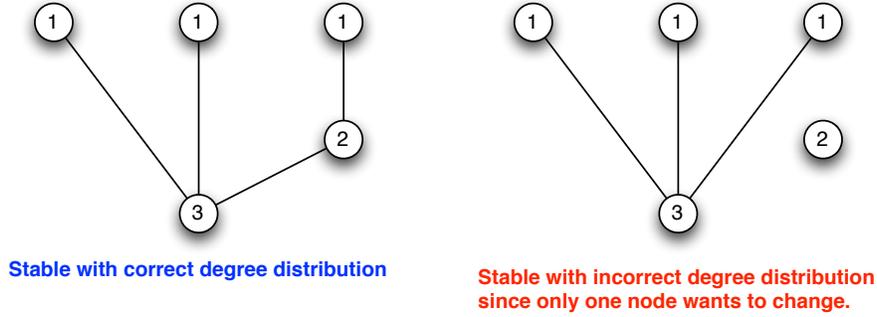}
\caption{We illustrate two graphs that are both pairwise stable for degree sequence $\mathbf{d} = (1,1,1,2,3)$ with $f(x) = x^2$ in Theorem \ref{thm: general arbitrary dist}}.
\label{fig:CounterExample}
\end{figure}
It should be clear however, that the second graph is not a \textit{Pareto optimal} solution to the problem. Pareto optimality of the class of graphs $\eta^{-1}(\mathbf{d})$ follows \textit{a posteriori} from the proof of Theorem \ref{thm: general arbitrary dist}.
\end{remark}
\begin{corollary} Assume the conditions from Theorem \ref{thm: general arbitrary dist}. A graph $g$ is a pairwise stable and Pareto optimal solution to the game $\mathcal{G}(v,Y,N)$ if and only if $\eta(g) = \mathbf{d}$.
\end{corollary}
\begin{proof} The fact that $\mathbf{d}$ is graphic implies that there is some graph $g \in \eta^{-1}(g)$ and furthermore $Y_i(g)$ is maximized for all $i = 1,\dots,n$. Therefore any change in strategy cannot improve any player's current payoff. Thus, $g$ is Pareto optimal and pairwise stable.

If $g' \not \in \eta^{-1}(\mathbf{d})$ then while $g'$ may be pairwise stable, it cannot be Pareto optimal since there is some other graph $g \in \eta^{-1}(\mathbf{d})$ so that $Y_i(g) > Y_i(g')$ for some $i \in N$.
\end{proof}

\begin{example}
Suppose that we want the degree distribution of a stable graph that results from playing the game described in Theorem \ref{thm: general arbitrary dist} to have a power law degree distribution.  We embed this into the objectives of the players, so the resulting graph has the proper distribution. Let $N=100$ players attempt to minimize their cost function
\begin{displaymath}
f(\eta_{i}(g))=(\eta_{i}(g) - k_{i})^{2}+\psi
\end{displaymath}
where $\psi=2$ and the parameters of each player ($k_i$) are given in the following table:
\begin{center}
\begin{tabular}{|c|c|}
  \hline
Node(s) & $k_{i}$ \\
\hline
1-75 & 1\\
\hline
76-89 & 2\\
\hline
90-94 & 3\\
\hline
95-96 & 4\\
  \hline
97 & 5\\
\hline
98 & 6\\
\hline
99 & 7\\
\hline
100 & 8\\
\hline
\end{tabular}
\end{center}
The distribution of values of $k_{i}$ forms an approximate (with rounding to integers) power law distribution.  This is illustrated in Figure \ref{fig:subfig2}:
\begin{figure}[htbp]
\centering
\includegraphics[scale=0.33]{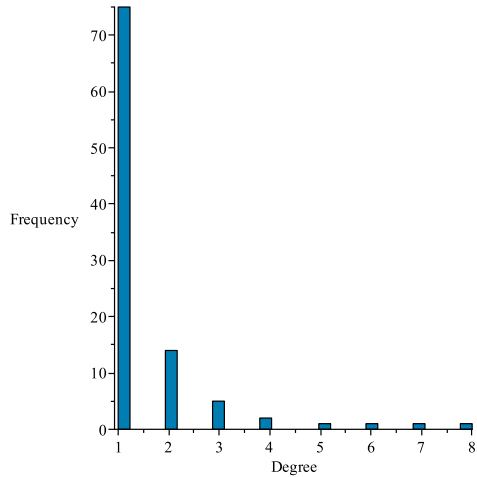}
\caption{The empirical desired distribution of the degrees of the players. This degree distribution follows an approximate power law distribution.}
\label{fig:subfig2}
\end{figure}

Figure \ref{fig:X2} illustrates one realization of a power law graph with the desired degree distribution.
\begin{figure}[ht]
\centering
\includegraphics[scale=0.25]{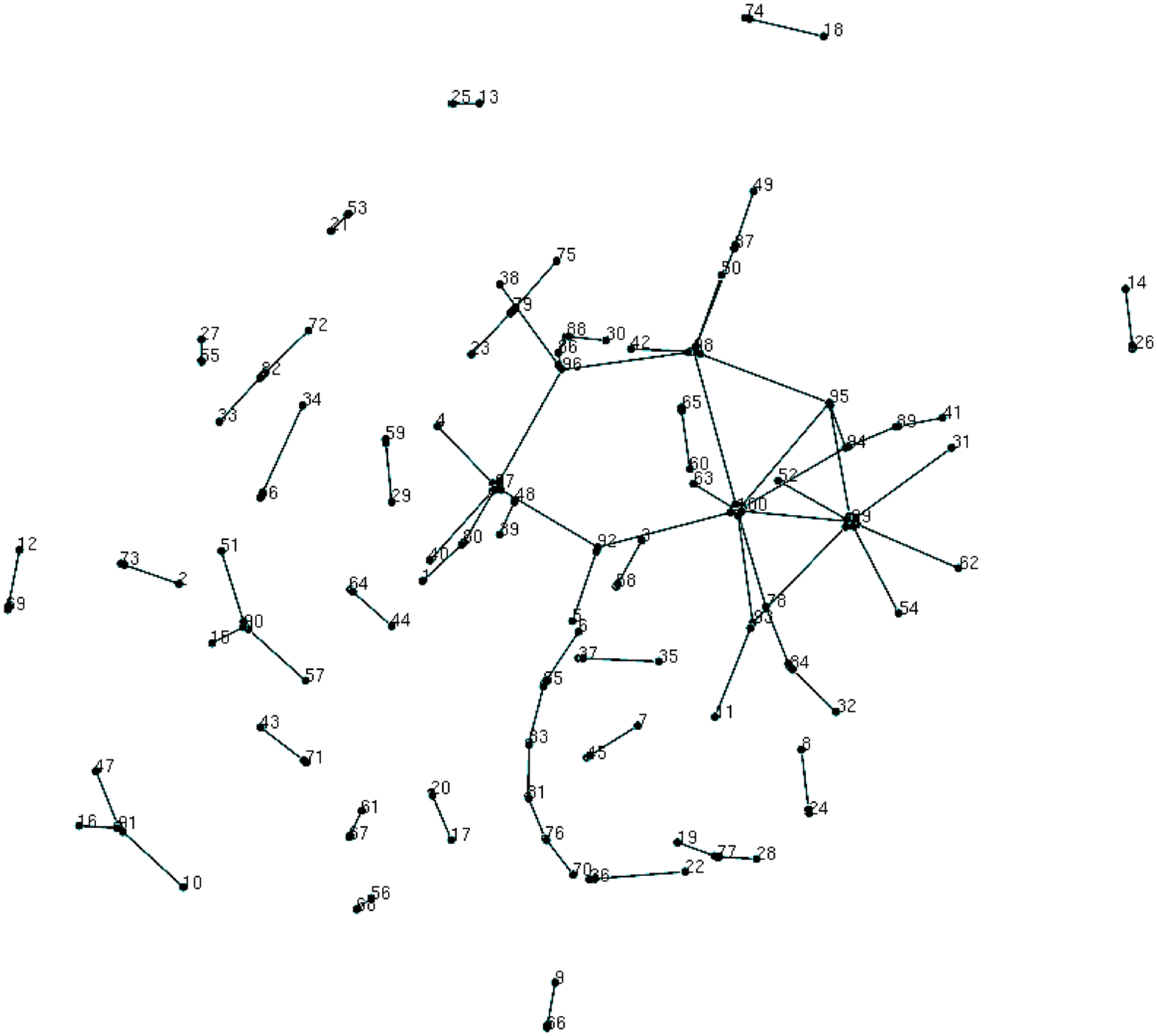}
\caption{Sample pairwise stable collaboration Network for 100 nodes.}
\label{fig:X2}
\end{figure}
We can test the stability of any graph $g$ with $\eta_{i}(g)=k_{i}$. In doing so, we observe that any graph with this distribution must be both stable (since there is no motivation for any player to alter her configuration within the network) and Pareto optimal.

We can experimentally analyze this simple game to determine the nature of the average network produced during play. For a simple experiment, 100 pairwise stable graphs were generated and some of their properties analyzed. To build these pairwise stable graphs, we used the following procedure:
\begin{enumerate*}
\item An empty graph was initialized.
\item Two vertices that had degree less than their desired degree were chosen at random without replacement.
\item These two vertices were joined by an edge.
\item The graph was checked for pairwise stability. If the graph was pairwise stable, the graph was returned. Otherwise, we continued at Step 2.
\end{enumerate*}
Results are shown in Figure \ref{fig:ExperimentalResults}.
\begin{figure}[htbp]
\centering
\subfigure[]{\includegraphics[scale=0.25]{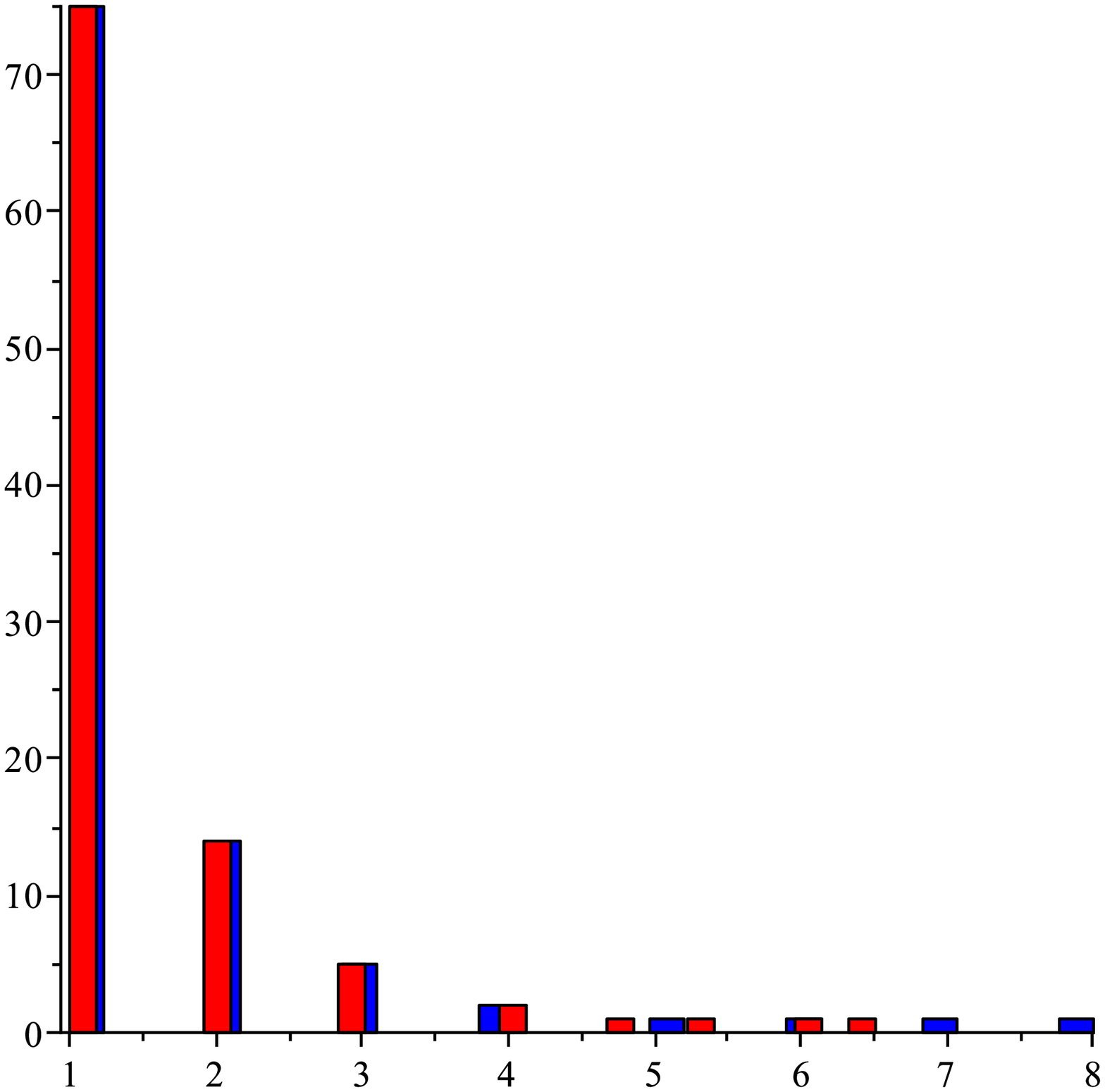}}
\subfigure[]{\includegraphics[scale=0.25]{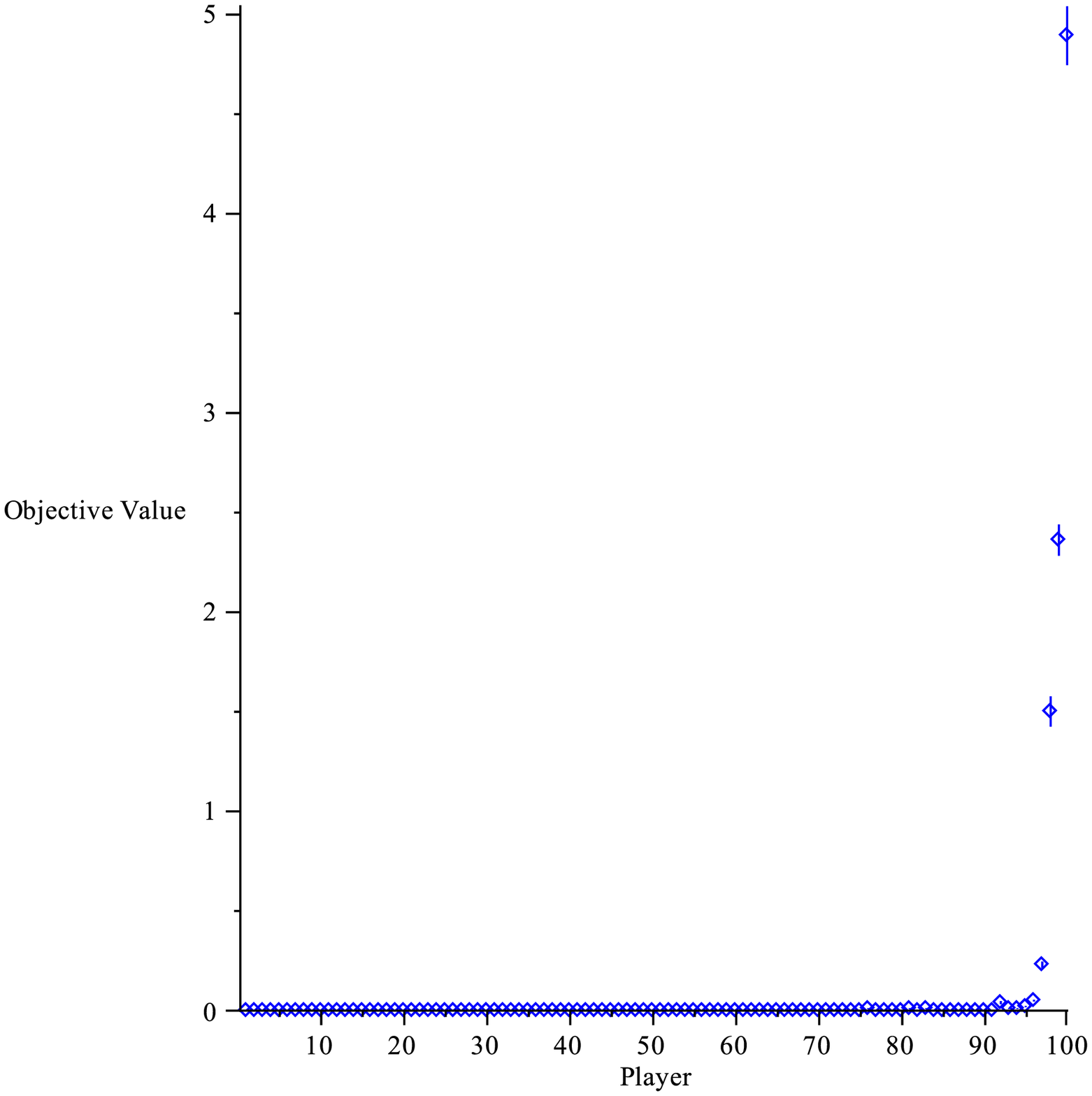}}
\caption{Experimental results from 100 pairwise stable graphs that result when the given degree distribution is used with the objective $f(\eta_{i}(g))=(\eta_{i}(g) - k_{i})^{2}+\psi$.}
\label{fig:ExperimentalResults}
\end{figure}
The left histogram shows the similarity between the degree specified degree distribution and the mean degree distribution of the ensemble of graphs generated. A more interesting plot is the mean objective function value on a per-player basis (with error bars). This graph indicates that most players successfully minimize their objective value (to zero) and that the players with the worst ability to optimize their objective functions are players with high degree requirements. It is interesting to note that these results are a function of the routine used to compute pairwise stable graphs. We also tried using a modified version of the preferential attachment algorithm:
\begin{enumerate*}
\item An empty graph was initialized.
\item Two vertices that had degree less than their desired degree and with highest desired degree where chosen at random.
\item These two vertices were joined by an edge.
\item The graph was checked for pairwise stability. If the graph was pairwise stable, the graph was returned. Otherwise, we continued at Step 2.
\end{enumerate*}
Using this algorithm, in 100 simulation runs a graph with the proscribed degree sequence was \textit{always} returned\footnote{Maple code for these experiments is available upon request from the authors.}.
\end{example}

%

\section{Application to Collaboration in Oligopoly}\label{sec:AppOlig}
We present an application of the network formation game to firm collaboration in oligopolies, which is an extension to the firm collaboration presented by Goyal and Joshi in \cite{goyal2003}. We show the conditions under which a firm collaboration network with an arbitrary degree sequence will be stable.

The number of firms participating in collaborative agreements (e.g., sharing resources such as equipment, laboratory space, office space, engineers and scientists through separate R \& D subcompanies) has significantly increased  within industries that are R \& D intensive \cite{hagedoorn1990,hagedoorn1993,hagedoorn1994,hagedoorn1996,hagedoorn2000,hagedoorn2002}, sparking research investigating the incentives for such collaboration.

Goyal and Joshi examine the incentives for collaboration and the interaction of these incentives under market competition using a model of horizontal oligopolistic firm collaboration. In this model, firms compete in the market after choosing collaborators \cite{goyal2003}. Goyal and Joshi provide theoretical analysis on various models with varying levels of link formation costs relative to production costs. They assume that collaboration requires a fixed cost from each firm.  These models investigate collaboration agreements where the collaboration reduces the cost of production. This prior work assumes a constant return to scale cost function, which they admit is a restrictive assumption, although it has been made by others in the collaborative R \& D literature \cite{bloch1995}.

\subsection{Previous Results on Network Stability in Oligopoly}
Following \cite{goyal2003}, consider an oligopoly composed of $n$ firms, each of whom may collaborate with any of the other $n-1$ firms. Firm $i$ produces a quantity $q_i$ of a product. Collaboration among firms affects the marginal cost of production. Thus a particular (collaboration) graph $g$ induces a marginal cost vector $c(g) \in \mathbb{R}^n$ in which $c_{i}(g)$ is the marginal cost of firm $i$ under collaboration graph $g$. It is assumed that the marginal cost decreases for a firm $i$ as the number of collaborators increase.  For many of the models in \cite{goyal2003}, it is assumed that the marginal cost of firm $i$ linearly decreases with the number of collaborators for firm $i$:
\begin{equation}\label{eq: marginal cost}
c_{i}(g)=\gamma_{0}-\gamma \eta_{i}(g)
\end{equation}
where, as before, $\eta_{i}(g)$ is the number of links for firm $i$ and $\gamma_{0}$ is the marginal cost of production when a firm has no links.  Notice that $\gamma_{0}$ is constant for all firms.  Given a network $g$, there is an induced set of costs which, along with the demand functions, produces a set of profit functions for each firm, $Y_{i}(g)$ (the allocation of payoff for player $i$).  These profit functions then induce a Nash equilibrium of production, which provides the precise allocation rule (i.e., profit) for each firm on the graph.  The stability of the collaboration network can then be analyzed as in Section \ref{sec: Beyond Oligopolies}.

One example that Goyal and Joshi \cite{goyal2003} study is that of a homogenous product oligopoly.  The market marginal price function (dependent on quantity produced) is given by
\begin{equation}\label{eq: market demand}
    P(Q) = \alpha - \sum_{i \in N} q_{i}
\end{equation}
where $Q = \sum_{i \in N} q_i$. Using Equation \ref{eq: marginal cost}, the resulting profit to Player $i$ is:
\begin{equation}
Y_i(g) =
\left(\alpha - \sum_{i \in N} q_{i}\right) q_i -
\left(\gamma_{0}-\gamma \eta_{i}(g)\right)q_i =
(\alpha - \gamma_0)q_i - \left(\sum_{i \in N} q_{i}\right)q_i -
\left(-\gamma \eta_{i}(g)\right)q_i
\label{eqn:Profit1}
\end{equation}

Given a collaboration network $g$ and marginal cost and demand induced by network $g$ we can use Expression \ref{eqn:Profit1} to find the quantity produced by each firm at Cournot equilibrium using the standard Cournot oligopoly production formulation \cite{Tir88, goyal2003}:
\begin{equation}\label{eq: quantity produced}
q_{i} = \frac{\alpha - \gamma_{0}+n \gamma \eta_{i}(g)- \gamma \sum_{j \neq i} \eta_{j}(g)}{n+1}
\end{equation}
The fact that $q_i$ is a function of the collaboration graph $g$ implies that the quantities: $P$, $Q$ and $c_i$ are all functions of $g$. We will refer to these quantities in this way in the sequel.

There is no natural market force in this model to ensure that a firm does not produce a negative quantity.  Instead, Goyal and Joshi restrict the parameters by looking at the case where firm $i$ has no collaborators ($\eta_{i}(g)=0$) and all others have a maximum number of collaborators ($\eta_{j}(g)=n-2$ for all $j \neq i$) and forcing the quantity production to remain non-negative; i.e.:
\begin{equation}
\alpha - \gamma_{0}- \gamma (n-1) (n-2) \geq 0
\end{equation}
Goyal and Joshi show that with marginal cost (\ref{eq: marginal cost}) and market demand (\ref{eq: market demand}), the complete network is the unique stable network \cite{goyal2003}. They point out that the positive contribution to profit obtained by a collaboration link is on the order of $n\gamma$ and the negative contribution to profit is on the order of $\gamma$, so the net profit is on the order of $(n-1)\gamma$, which leads to a firm's incentive to saturate all of its possible links, thus making the complete graph stable.

\begin{example} We present a numerical example of the result from \cite{goyal2003}. Let $N=5$ firms compete in an oligopoly with inverse demand function $P=100-Q$, fixed cost $\gamma_{0}=5$, and $\gamma = 1$. Then we can see that:
\begin{equation}
\alpha - \gamma_{0}- \gamma (n-1) (n-2) = 100 - 5 - (4)(3) = 95 - 12 = 83 > 0
\end{equation}
This enforces the non-negativity of production. Since the marginal cost of Player $i$ ($i = 1,\dots,5$) decreases as its degree increases (from Equation \ref{eq: marginal cost}) it is easy to see that players will prefer as many links as possible. Thus, the complete graph must be stable.
\end{example}

\begin{remark} The fundamental limitation to Goyal and Joshi's work lies in the model complexity. Maintaining collaborations is costly and each additional collaboration may contribute less to the marginal cost than the collaboration before it. Additionally, Goyal and Joshi's model admits only complete collaboration networks, omitting any other type of collaboration network. The remainder of the paper illustrates how to generalize this result to graphs with arbitrary degree sequences.
\end{remark}

\subsection{Results of Nonlinear Cost on Stability}
In this section we study the effect a nonlinear variation on the marginal cost function has on the stability of collaboration structures. We consider a marginal cost function:
\begin{equation}
c_i(g) = \gamma_0 + f_i(\eta_i(g))
\label{eqn:MarginalCostf}
\end{equation}
where $f_{i}$ is some function $f_{i}:\mathbb{R} \rightarrow \mathbb{R}$. We present two Theorems:
\begin{enumerate*}
\item  Theorem \ref{thm: negative dv dt complete graph} extends the results from Goyal and Joshi \cite{goyal2003} to the case where firms receive decreasing marginal benefit from collaboration via a nonlinear decrease in production cost.  This extension allows the parameter space to be considerably less restricted than the result in \cite{goyal2003}.

\item Theorem \ref{thm: stable assym k distn graph} shows that if the functions $f_i$ have minima and the parameter space is restricted to keep production quantities nonnegative, then the collaboration network in the oligopolistic competition model will have a graph corresponding to the arbitrary degree sequence in the prior section.
\end{enumerate*}
As a result, this is a particular application of Theorem \ref{thm: general arbitrary dist}, to a  general collaboration game. To prove our theorems, we alter the assumptions on the functions $f_i$ from the marginal cost (Expression \ref{eqn:MarginalCostf}).  Theorem \ref{thm: stable assym k distn graph} requires convex functions $f_i$ where as Theorem \ref{thm: negative dv dt complete graph} requires $f_i$ to be a decreasing and convex function.

\begin{lemma}
\label{lem: qi}
Suppose we have an oligopoly consisting of $n$ firms in which collaboration is defined by the graph $g$ and the profit function (allocation rule) for Firm $i$ in that oligopoly is given by:
\begin{equation}
Y_{i}(g,q(g)) = (\alpha - \gamma_0)q_{i}(g) - \left(\sum_{j \in N} q_{j}\right)q_i(g) - (f_{i}(\eta_{i}(g)))q_i(g)
\label{eqn: profit oligopoly}
\end{equation}
then the quantity produced for firm $i$ is:
\begin{equation}
\label{eq: quantity produced app1}
q_{i}(g) = \frac{\alpha - \gamma_{0}-n f_{i}(\eta_{i}(g))+ \sum_{j \neq i}f_{j}(\eta_{j}(g)) }{n+1}
\end{equation}
\end{lemma}
\begin{proof} From \cite{Tir88}, for any oligopoly with profit function of the form:
\begin{equation}
Y_i(q) = aq_i - \left(\sum_{j \in N} q_{j}\right)q_i - b_iq_i
\end{equation}
The resulting Cournot equilibrium point on quantities is:
\begin{equation}
q_i = \frac{a - nb_i + \sum_{j \neq i}b_j}{n+1}
\label{eqn:CournotEq}
\end{equation}
In our case, we have:
\begin{gather*}
a = \alpha - \gamma_0\\
b_i = f_{i}(\eta_{i}(g)) \quad \forall i
\end{gather*}
Substituting these definitions into Expression (\ref{eqn:CournotEq}) yields Expression (\ref{eq: quantity produced app1}). This completes the proof.
\end{proof}

\begin{remark}
It is worth noting that when for each firm $i$, $b_i = -\gamma \eta_{i}(g)$ then the cost function (\ref{eq: marginal cost}) and induced equilibrium quantity (\ref{eq: quantity produced}) is retrieved from Goyal and Joshi.
\end{remark}

\begin{corollary}
\label{cor: qi nonnegative}
Suppose that $f$ is a convex function that has a minimum at $0$.  Further, suppose $f_{i}(\eta_{i}(g))=f(\eta_{i}(g)-k_{i})$ where $k_{i} \in \{1,2,\ldots ,n-1\}$.
If the parameters $\alpha$ and $\gamma_{0}$ and the function $f$ are such that:
\begin{equation}
\label{eq: qi nonnegative}
\alpha - \gamma_{0}-n \max(f(n-1),f(1-n))-\frac{1}{2}(n-1)\max(f(1)-f(0),f(-1)-f(0)) > 0
\end{equation}
and $n \geq 2$, then the Cournot equilibrium point quantities (\ref{eq: quantity produced app1}) are nonnegative for all firms and for all collaboration graphs and the two inequalities are true:
\begin{align}
2q_{i}(g)-\frac{n-1}{n+1} [f(1)-f(0)] &>0 \label{eq: ineq 1} \\
2q_{i}(g)-\frac{n-1}{n+1}[f(-1)-f(0)] &>0  \label{eq: ineq 2}
\end{align}
\end{corollary}

\begin{proof}
Since $n \geq 2$ and $f$ is convex and has a minimum at $0$, this implies that $\frac{n-1}{n+1} [f(1)-f(0)]$ and $\frac{n-1}{n+1} [f(-1)-f(0)]$ are non-negative.  Hence, (\ref{eq: ineq 1}) and  (\ref{eq: ineq 2}) imply that $q_{i}(g)$ is non-negative and hence it suffices to only show that (\ref{eq: ineq 1}) and  (\ref{eq: ineq 2}) are implied by (\ref{eq: qi nonnegative}).

For all $i$, function $f_{i}$ is a convex function of the degree of node $i$ in the graph $g$; the degree of node $i$ must take an integer value between $0$ and $n-1$, which due to the convexity of $f_{i}$ and the fact that $k_{i} \in \{1,2,\ldots n-1\}$ implies that the maximum of $f_{i}$ is equivalent to $\max(f(n-1),f(-n+1))$.  That is,
\begin{equation*}
f_{i}(\eta_{i}(g)) \leq \max(f(n-1),f(-n+1))
\end{equation*}
This means that (\ref{eq: qi nonnegative}) implies:
\begin{equation}
\label{eq: ineq 3}
\alpha - \gamma_{0}-n f_{i}(\eta_{i}(g))-\frac{1}{2}(n-1)\max(f(1)-f(0),f(-1)-f(0)) > 0 \quad \forall i
\end{equation}
Since, all $f_{i}(\eta_{i}(g)) \geq 0$, we may add $\sum_{j \neq i}f_{j}(\eta_{j}(g))$ to the left side of  (\ref{eq: ineq 3}) without harming the inequality, implying:
\begin{equation}
\label{eq: ineq 4}
\alpha - \gamma_{0}-n f_{i}(\eta_{i}(g))+\sum_{j \neq i}f_{j}(\eta_{j}(g))-\frac{1}{2}(n-1)\max(f(1)-f(0),f(-1)-f(0)) > 0 \quad \forall i
\end{equation}
Now we divide by $n+1$:
\begin{equation}
\label{eq: ineq 5}
\frac{\alpha - \gamma_{0}-n f_{i}(\eta_{i}(g))+\sum_{j \neq i}f_{j}(\eta_{j}(g))}{n+1}-\frac{1}{2}\frac{n-1}{n+1}\max(f(1)-f(0),f(-1)-f(0)) > 0 \quad \forall i
\end{equation}
The term on the left of (\ref{eq: ineq 5}) is $q_{i}(g)$:
\begin{equation}
\label{eq: ineq 6}
q_{i}(g)-\frac{1}{2}\frac{n-1}{n+1}\max(f(1)-f(0),f(-1)-f(0)) > 0 \quad \forall i
\end{equation}
Multiply through by two and note:
\begin{align*}
2q_{i}(g)>\frac{n-1}{n+1}\max(f(1)-f(0),f(-1)-f(0)) &>\frac{n-1}{n+1}[f(1)-f(0)]  \quad \forall i\\
2q_{i}(g)>\frac{n-1}{n+1}\max(f(1)-f(0),f(-1)-f(0)) &>\frac{n-1}{n+1}[f(-1)-f(0)]  \quad \forall i
\end{align*}
Now (\ref{eq: ineq 1}) and  (\ref{eq: ineq 2}) immediately follow.
\end{proof}
\begin{remark}
This essentially means that the steeper a function $f$ around zero and on the interval $(-n+1,n-1)$, the greater the quantity $\alpha-\gamma_0$ is needed to ensure the theorem proved later in this section.  It is worth pointing out that this bound may often not be tight (i.e., the inequalities may hold true and production quantities may be positive even when the condition is not not met).
\end{remark}

\begin{theorem}
\label{thm: stable assym k distn graph}
Suppose that $f$ is a convex function that has a minimum at $0$.  Further, suppose $f_{i}(\eta_{i}(g))=f(\eta_{i}(g)-k_{i})$.  Define the change in $f$ as $\triangle^{-} f_{i}(k_{i})=f_{i}(k_{i}-1)-f_{i}(k_{i})=f(-1)-f(0)=\triangle^{-} f(0)$ and $\triangle^{+} f_{i}(k_{i})=f_{i}(k_{i}+1)-f_{i}(k_{i})=f(1)-f(0)=\triangle^{+} f(0)$.  Suppose $n \geq 2$ firms compete in an oligopoly with market demand $p=\alpha - \sum_{i \in N}q_{i}$ and marginal costs $c_i(g) = \gamma_0 + f_i(\eta_i(g))$.  If the parameters $\alpha$ and $\gamma_{0}$ and the function $f$ obey condition (\ref{eq: qi nonnegative}), then the equivalence class of graphs $[g]_{\eta}$ such that $\eta_{i}(g)=k_{i}$ is an equivalence class of stable collaboration graphs.
\end{theorem}
\begin{proof}
Let $g$ be a graph $g$ such that $\eta_{i}(g)=k_{i}$ for all firms $i$.  Consider a firm $i$ who may consider dropping its link with node $j$.  If node $i$ drops its link with node $j$ leading to graph $g-ij$, then $\eta_{i}(g-ij)=k_{i}-1$ and $\eta_{j}(g-ij)=k_{j}-1$, while $\eta_{r}(g-ij)=k_{r}$ for $r \not \in \{i,j\}$.  
Using Lemma \ref{lem: qi}
\begin{equation}
q_{i} = \frac{\alpha - \gamma_{0}-n f_{i}(\eta_{i}(g))+ \sum_{j \neq i \in N}f_{j}(\eta_{j}(g)) }{n+1}
\end{equation}
Calculate:
\begin{align*}
q_{i}(g-ij) &= q_{i}(g)- \triangle^{-} f_{i}(k_{i}) \left( \frac{n}{n+1} \right)+ \triangle^{-} f_{j}(k_{j}) \left( \frac{1}{n+1}  \right)\\
q_{j}(g-ij) &= q_{j}(g)- \triangle^{-} f_{j}(k_{j}) \left( \frac{n}{n+1} \right)+ \triangle^{-} f_{i}(k_{i}) \left( \frac{1}{n+1}  \right)\\
q_{r}(g-ij) &= q_{r}(g)+ \triangle^{-} f_{i}(k_{i}) \left( \frac{1}{n+1} \right)+ \triangle^{-} f_{j} (k_{j})\left( \frac{1}{n+1}  \right)
\end{align*}
It then follows that
\begin{align*}
Q(g-ij)&=Q(g)- \left( \frac{1}{n+1} \right) (\triangle^{-} f_{i}(k_{i})+\triangle^{-} f_{j}(k_{j}))\\
P(g-ij)&=P(g)+\left( \frac{1}{n+1} \right)(\triangle^{-} f_{i}(k_{i})+\triangle^{-} f_{j}(k_{j}))  \\
c_{i}(g-ij)&=c_{i}(g)+\triangle^{-} f_{i}(k_{i})
\end{align*}
Now, we can calculate $Y_{i}(g-ij)$ in terms of $Y_{i}(g)$:
\begin{align*}
Y_{i}(g-ij)&=q_{i}(g-ij)[P(g-ij)-c_{i}(g-ij)] \\
   &= Y_{i}(g)+ q_{i}(g) \left( \frac{2}{n+1} \right) [  \triangle^{-} f_{j}(k_{j}) - n \triangle^{-} f_{i}(k_{i})]
   +\left( \frac{[  \triangle^{-} f_{j}(k_{j}) - n \triangle^{-} f_{i}(k_{i})]}{n+1} \right)^2
\end{align*}
Since $f_{i}(\eta_{i}(g))=f(\eta_{i}(g)-k_{i})$ this implies that $\triangle^{-} f_{i}(k_{i})=\triangle^{-} f_{j}(k_{j})$ leading to (\ref{eq: profit change 1A}) and then (\ref{eq: profit change 2A}) and (\ref{eq: profit change 3A}) through algebraic manipulation.  Finally, by the assumptions of the theorem and condition (\ref{eq: qi nonnegative}) each of the quantities $\triangle^{-} f_{i}(k_{i})$, $\frac{n-1}{n+1}$, and $2 q_{i}(g) -  \frac{n-1}{n+1} \triangle^{-} f_{i}(k_{i})$ are nonnegative implying (\ref{eq: profit change 4A}).
\begin{align}
Y_{i}(g-ij)-Y_{i}(g)&=2 q_{i}(g)\triangle^{-} f_{i}(k_{i}) \left( \frac{1-n}{n+1} \right)+(\triangle^{-} f_{i}(k_{i}))^{2} \left( \frac{1-n}{n+1} \right)^2 \label{eq: profit change 1A}\\
&=  \triangle^{-} f_{i}(k_{i}) \left(  \frac{1-n}{n+1}   \right) \left(2 q_{i}(g) + \frac{1-n}{n+1} \triangle^{-} f_{i}(k_{i})    \label{eq: profit change 2A} \right)\\
&= - \triangle^{-} f_{i}(k_{i}) \left(  \frac{n-1}{n+1}   \right) \left(2 q_{i}(g) -  \frac{n-1}{n+1} \triangle^{-} f_{i}(k_{i})   \label{eq: profit change 3A} \right)\\
&<0 \label{eq: profit change 4A}
\end{align}
This implies that if firm $i$ attempts to drop link $ij$, then $Y_{i}(g) > Y_{i}(g-ij)$ and thus firm $i$ decreases its profit.  The same will be true for firm $j$.  Hence, no firm has an incentive to drop a link from graph $g$.
Now, we will consider the case where firm $i$ attempts to add a link to the graph $g$, giving $g+ij$ under the assumption that the link $ij$ does not exist in graph $g$.  This analysis will follow closely the analysis for $g-ij$.  First note that $\eta_{i}(g)=k_{i}$ for all firms $i$ and $\eta_{i}(g+ij)=k_{i}+1$ and $\eta_{j}(g+ij)=k_{j}+1$, while $\eta_{r}(g+ij)=k_{r}$ for $r \not \in \{i,j\}$.  We define $\triangle^{+} f_{i}(k_{i})$ as $\triangle^{+} f_{i}(k_{i})=f_{i}(k+1)-f_{i}(k)$; note the subtle difference from the definition of $\triangle^{-} f_{i}(k_{i})$.  Again using Lemma \ref{lem: qi}, we calculate the production quantity for each node in graph $g+ij$:
\begin{align*}
q_{i}(g+ij) &= q_{i}(g)- \triangle^{+} f_{i}(k_{i}) \left( \frac{n}{n+1} \right)+ \triangle^{+} f_{j}(k_{j}) \left( \frac{1}{n+1}  \right)\\
q_{j}(g+ij) &= q_{j}(g)- \triangle^{+} f_{j}(k_{j}) \left( \frac{n}{n+1} \right)+ \triangle^{+} f_{i}(k_{i}) \left( \frac{1}{n+1}  \right)\\
q_{r}(g+ij) &= q_{r}(g)+ \triangle^{+} f_{i}(k_{i}) \left( \frac{1}{n+1} \right)+ \triangle^{+} f_{j}(k_{j}) \left( \frac{1}{n+1}  \right)
\end{align*}
We can then calculate the corresponding total production quantity $Q$, the market price $P$ and marginal costs for each player for the graph $g+ij$:
\begin{align*}
Q(g+ij)&=Q(g)- \left( \frac{1}{n+1} \right) (\triangle^{+} f_{i}(k_{i})+\triangle^{+} f_{j}(k_{j}))\\
P(g+ij)&=P(g)+\left( \frac{1}{n+1} \right)(\triangle^{+} f_{i}(k_{i})+\triangle^{+} f_{j}(k_{j}))  \\
c_{i}(g+ij)&=c_{i}(g)+\triangle^{+} f_{i}(k_{i})
\end{align*}
Now, we can calculate $Y_{i}(g+ij)$ in terms of $Y_{i}(g)$:
\begin{align*}
Y_{i}(g+ij)&=q_{i}(g+ij)[P(g+ij)-c_{i}(g+ij)] \\
   &= Y_{i}(g)+ q_{i}(g) \left( \frac{2}{n+1} \right) [  \triangle^{+} f_{j}(k_{j}) - n \triangle^{+} f_{i}(k_{i})]
   +\left( \frac{[  \triangle^{+} f_{j}(k_{j}) - n \triangle^{+} f_{i}(k_{i})]}{n+1} \right)^2
\end{align*}
Since $f_{i}(\eta_{i}(g))=f(\eta_{i}(g)-k_{i})$ this implies that $\triangle^{+} f_{i}(k_{i})=\triangle^{+} f_{j}(k_{j})$ leading to (\ref{eq: profit change 5}) and then (\ref{eq: profit change 6}) and (\ref{eq: profit change 7}) through algebraic manipulation.  Finally, by the assumptions of the theorem and condition (\ref{eq: qi nonnegative}), each of the quantities $\triangle^{+} f_{i}(k_{i})$, $\frac{n-1}{n+1}$, and $2 q_{i}(g) -     \frac{n-1}{n+1} \triangle^{+} f_{i}(k_{i})$ are positive implying (\ref{eq: profit change 8}).
\begin{align}
Y_{i}(g+ij)-Y_{i}(g)&=2 q_{i}(g)\triangle^{+} f_{i}(k_{i}) \left( \frac{1-n}{n+1} \right)+(\triangle^{+} f_{i}(k_{i}))^{2} \left( \frac{1-n}{n+1} \right)^2 \label{eq: profit change 5}\\
&=  \triangle^{+} f_{i}(k_{i}) \left(  \frac{1-n}{n+1}   \right) \left(2 q_{i}(g) + \frac{1-n}{n+1} \triangle^{+} f_{i}(k_{i})    \label{eq: profit change 6} \right)\\
&= - \triangle^{+} f_{i}(k_{i}) \left(  \frac{n-1}{n+1}   \right) \left(2 q_{i}(g) -     \frac{n-1}{n+1} \triangle^{+} f_{i}(k_{i})   \label{eq: profit change 7} \right)\\
&<0 \label{eq: profit change 8}
\end{align}
This implies that if firm $i$ attempts to add a link $ij$, then $Y_{i}(g) > Y_{i}(g+ij)$ and the firm decreases its profit.  The same will be true for firm $j$.  Hence, no firm has an incentive to add a link to graph $g$. Since no firm has an incentive to add or drop a link to graph $g$, it is stable. This completes the proof.
\end{proof}


\begin{example}
\label{ex: assymetric oligopoly}
We present a numerical example of Theorem \ref{thm: stable assym k distn graph}. Let $N=5$ firms compete in an oligopoly with inverse demand function $P=100-Q$, fixed cost $\gamma_{0}=5$, and $f(\eta(g))=(\eta(g))^{2}+\psi$ where $\psi=2$.  Each firm has the shifted function $f_{i}(\eta_{i}(g))=f(\eta_{i}(g)-k_{i})=(\eta_{i}(g) - k_{i})^{2}+\psi$ where $\mathbf{k}=[2,3,4,3,2]^T$ and .  We want to test the stability of a graph $g$ with $\eta_{i}(g)=k_{i}$ and $f_{i}(\eta_{i}(g))=(\eta_i(g)-k_{i})^{2}+\psi$ for each node $i$. In order to apply theorem (\ref{thm: stable assym k distn graph}), we must ensure condition (\ref{eq: qi nonnegative}) is met.  We calculate:
\begin{gather*}
f(-n+1)=f(-4)=18,\quad
f(-1)=3,\quad
f(0)=2, \\
f(1)=3,\quad
f(n-1)=f(4)=18,\quad
\end{gather*}
Hence,
\begin{gather*}
\max(f(n-1),f(1-n))=18,\quad
\max(f(1)-f(0),f(-1)-f(0))=1,\quad
\\
n \max(f(n-1),f(1-n))=5\cdot18=90,\quad
(n-1)\max(f(1)-f(0),f(-1)-f(0))=4\cdot1=4,\quad
\end{gather*}

\begin{multline*}
\alpha - \gamma_{0}-n \max(f(n-1),f(1-n))-\frac{1}{2}(n-1)\max(f(1)-f(0),f(-1)-f(0))= \\100-5-90-\frac{1}{2}\cdot4=3>0
\end{multline*}
Two stable isomorphic graphs, shown in Figure \ref{fig:Stable}, have a degree sequence equivalent to $\mathbf{k}$. 
\begin{figure}[htbp]
\centering
\includegraphics[scale=0.5]{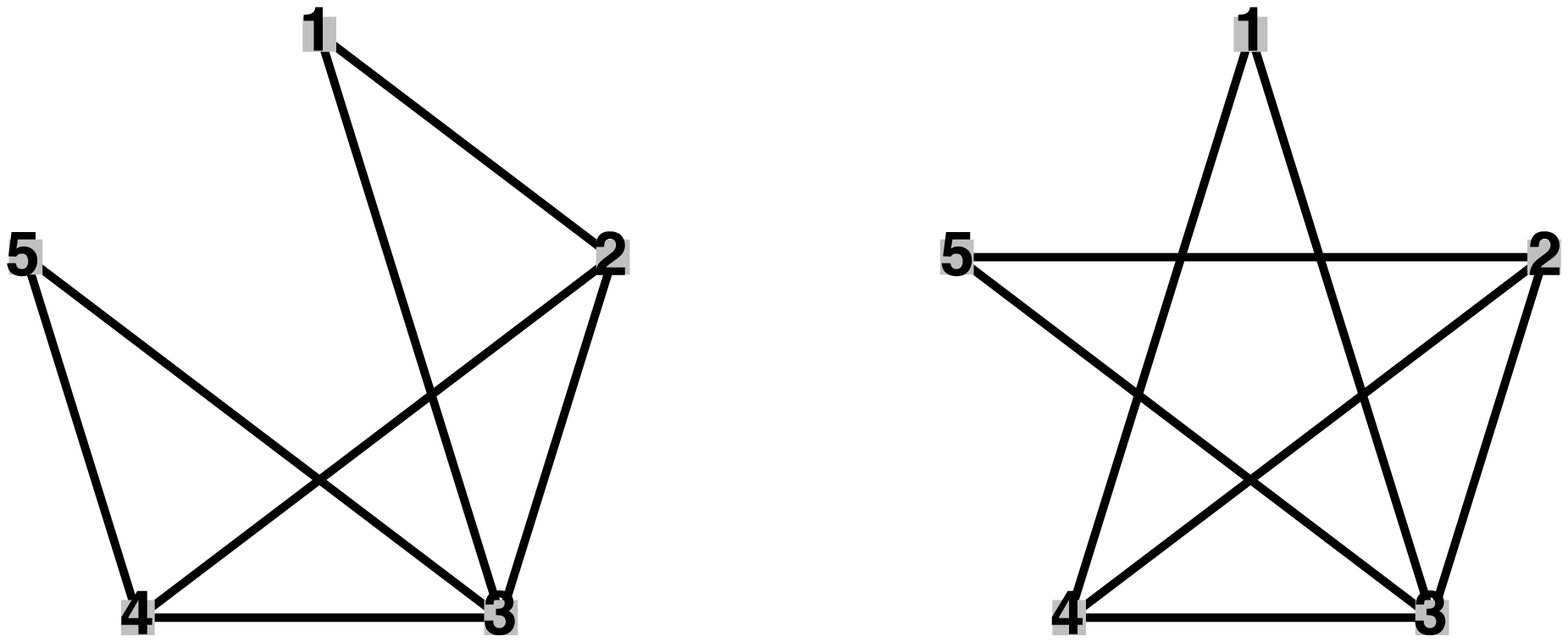}
\caption{Two stable and isomorphic graphs that are possible configurations for collaboration in the example oligopoly.}
\label{fig:Stable}
\end{figure}

However, for the given parameters, there are 33 stable graphs of which only the two shown have degree sequence equal to $\mathbf{k}$. These graphs were computed using Maple and are shown in Figure \ref{fig:AllStable}.
\begin{figure}[htbp]
\centering
\includegraphics[scale=0.5]{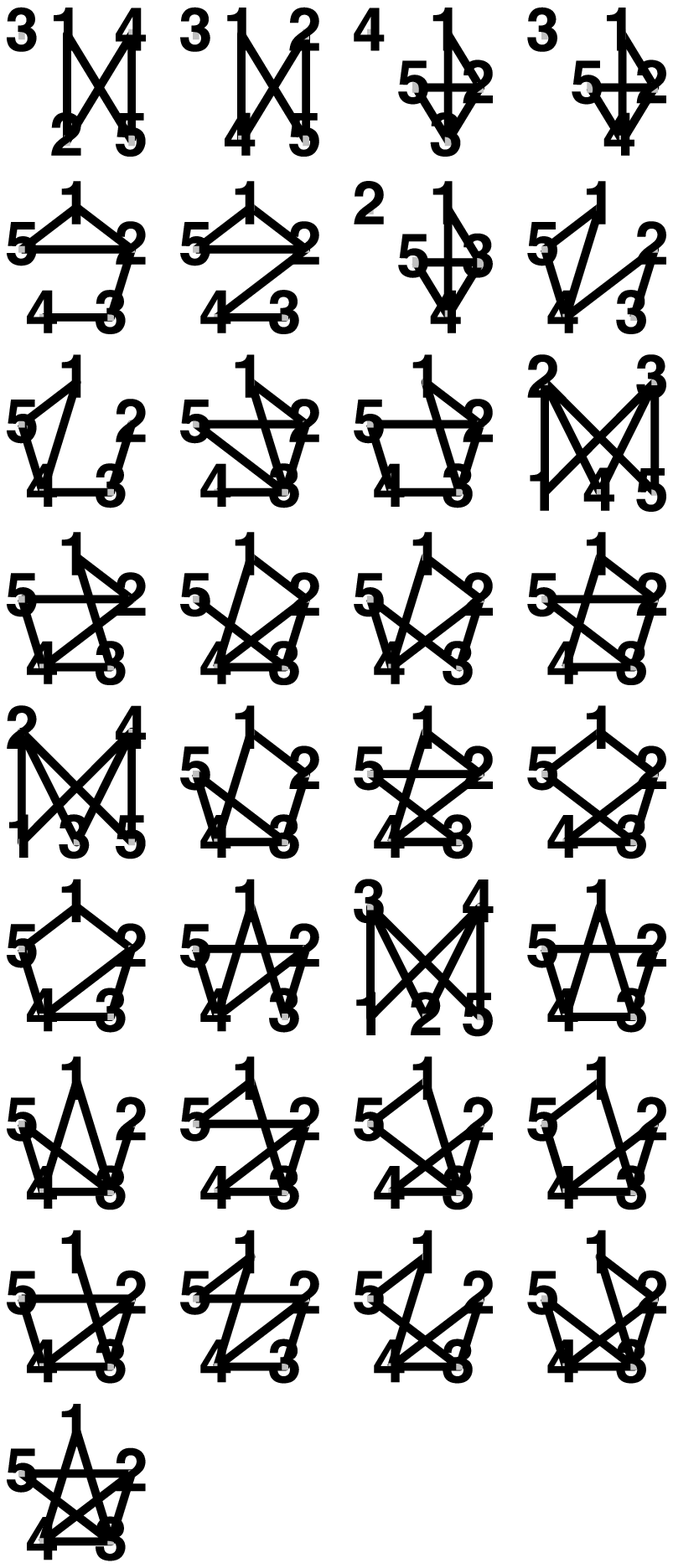}
\caption{The 29 stable graphs that arise from the parameters given.}
\label{fig:AllStable}
\end{figure}
It is interesting to investigate the other stable networks with degree sequences different from $\mathbf{k}$.  The graph in Row 7, Column 2 of Figure \ref{fig:AllStable} shows a stable configuration where both nodes 3 and 4 would prefer one more link, but they are already linked together and no other node requires an additional link.  Hence, the network is stable because each node is either satisfied or no pair of nodes can bilaterally improve themselves through the addition or removal of a link.  It is interesting to note that in the case of the graph shown, if both nodes 1 and 5, were to give up their link with one another and instead link to nodes 4 and 3 respectively, then each node would have minimized its marginal costs.  While it is the case that nodes 1 and 5 would not benefit from such a trade that would help nodes 3 and 4. These nodes would indirectly hurt themselves via the decreased market price as a result of the additional quantity produced by nodes 3 and 4.  Nonetheless this analysis brings out the fact that the manner in which nodes link to one another and the manner in which stability is analyzed, greatly affects which networks are deemed to be stable.
\end{example}

\begin{example} For small numbers of firms $<7$ generating the set of stable graphs takes seconds (even in an interpreted language like Maple).  
Consider the case of a firm (in this case Firm 3) who determines that collaboration is desired, but recognizes there may be a sink cost (not included in the model) associated with initializing such a collaboration. For example, the time taken to establish industry connections will cost  in terms of human labor. Assuming Firm 3 is a selfish profit maximizer and assumes that all other firms are selfish profit maximizers who will play to a stable configuration, then the analysis of the potentially stable graphs will inform Firm 3 on the potential payoffs it might receive. If the network is already in a stable configuration, then since there are three stable configurations in which Firm 3 is not connected to any other players, then it may not be worthwhile to even explore collaboration. On the other hand, if there is no collaboration (an unstable condition), then Firm 3 may hope to steer the network evolution and will evaluate its various payoffs in each possible stable configuration. The possible payoffs are shown in Figure \ref{fig:Player3Payoff}:
\begin{figure}[htbp]
\centering
\includegraphics[scale=0.5]{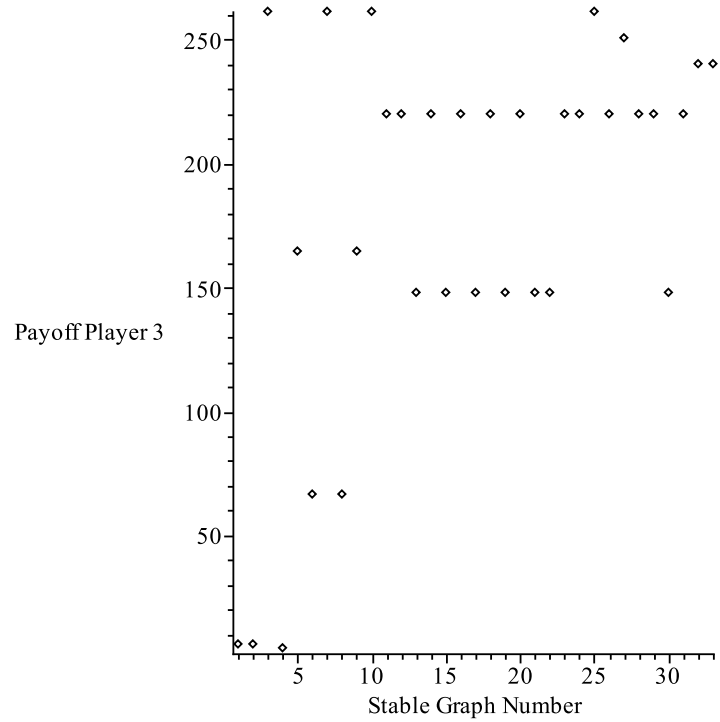}
\caption{The possible payoffs Player 3 can obtain in the various stable configurations. From Figure \ref{fig:AllStable}, graphs are numbered from left-to-right, top-to-bottom.}
\label{fig:Player3Payoff}
\end{figure}
Note that there is a substantial variation in the payoff the Player 3 may receive. In the empty graph, Player 3 receives a payoff of $169/4$. Notice that some collaborative scenarios disadvantage Player 3 ($3$ of the $33$) while most improve Player 3's payoff. 
\end{example}

\begin{remark}
Goyal and Joshi show that in a market with homogenous products and under quantity competition, the uniquely stable network is the complete network for specific cost functions under some parameter restrictions \cite{goyal2003}.  Theorem \ref{thm: negative dv dt complete graph} extends the result in Goyal and Joshi \cite{goyal2003} by using nonlinear costs and decreasing the restriction on parameters necessary for the model to be feasible.
\end{remark}

\begin{theorem}
\label{thm: negative dv dt complete graph}
Suppose the marginal cost for firm $i$ on graph $g$ is $c_{i}(g)=\gamma_{0}+f(\eta_{i}(g))$.  Define $\triangle f(k)=f(k+1)-f(k)$.  Further suppose each firm reduces its marginal cost for each additional collaboration (Condition (1)), the collaboration cost reduction function $f$ is convex (Condition (2)), collaboration cost reduction function $f$ is positive (Condition (3)), the production quantity is positive (Condition (4)), and the collaboration cost reduction function $f$ is not too steep (Condition (5)):
\begin{enumerate*}
\item $x_{1}>x_{2}$ implies $f(x_{1})<f(x_{2})$; \\
\item $f$ is convex; \\
\item $f$ is positive; \\
\item $\alpha - \gamma_{0}> n f(0)$; and\\
\item $\triangle f(k_{1})-n \triangle f(k_{2}) > 0$ for all $k_{1},k_{2} \in \{0,1,\ldots n-1\}$
\end{enumerate*}
Then the complete network $g^{c}$ is a stable graph.
\end{theorem}

\begin{proof}
The production quantity and profit can be calculated as before in (\ref{eq: quantity produced app1}) and (\ref{eqn: profit oligopoly}), respectively, by using that fact that $f_{i}=f$ for all $i$.
\begin{align*}
    q_{i}(g) &= \frac{\alpha - \gamma_{0}-n f(\eta_{i}(g))+ \sum_{j \neq i \in N}f(\eta_{j}(g)) }{n+1}\\
    Y_{i}(g)&=q_{i}(g)\left( P(Q)-c_{i}(g) \right)=q_{i}(g)\left( \alpha - \left(\sum_{j}q_{j}(g)\right)-\gamma_{0}-f(\eta_{i}(g))  \right)\\
\end{align*}

Observe through algebraic manipulation that $P(Q)-c_{i}(g)=q_{i}(g)$, which implies that $Y_{i}(g)=q_{i}(g) \left( P(Q)-c_{i}(g) \right)=\left( q_{i}(g) \right)^{2}$.%

For a complete network $g^{c}$, $\eta_{j}(g)=n-1$ for all $j$.  Given $g^{c}$:
\begin{equation*}
q_{i}(g^{c}) =\frac{\alpha - \gamma_{0}- f(n-1) }{n+1}
\end{equation*}
Now for a complete network missing a link $(i,j)$, denoted as $g = g^{c}-ij$, for nodes $i$ and $j$, $\eta_{i}(g)=\eta_{j}(g)=n-2$ but all other nodes $k$ have $\eta_{k}(g)=n-1$. Thus:
\begin{equation*}
q_{i}(g^{c}-ij) = \frac{\alpha - \gamma_{0}+(n-2)f(n-1)-(n-1)f(n-2)}{n+1}
\end{equation*}
It follows that:
\begin{equation*}
q_{i}(g^{c})-q_{i}(g^{c}-ij) =\frac{(n-1) [f(n-2)- f(n-1)]}{n+1}
\end{equation*}
Since $n-1 > n-2$, Condition (1) implies $f(n-1)<f(n-2)$ which implies $f(n-2)- f(n-1)>0$.  Hence, $q_{i}(g^{c})-q_{i}(g^{c}-ij)  > 0$.
By Condition (4), $\alpha - \gamma_{0}>n f(0)$, which implies $q_{i} > 0$ by Corollary \ref{cor: qi nonnegative} using $f_{i}=f$.  Further, since $q_{i}(g^{c})+q_{i}(g^{c}-ij) > 0$, which combined with $q_{i}(g^{c})-q_{i}(g^{c}-ij) > 0$, implies $Y_{i}(g^{c})-Y_{i}(g^{c}-ij)>0$.  This implies that any firm $i$ will decrease its profit by dropping link $ij$, hence the graph $g^{c}$ is stable.  Now to show that it is the only stable graph, suppose there is another graph $g \neq g^{c}$ that is also stable.  This implies that there exists a pair of firms $(i,j)$ such that $ij \not\in g$.  Let us consider the graph $g+ij$ relative to the graph $g$.
\begin{displaymath}
q_{i}(g+ij)-q_{i}(g)=\frac{-n}{n+1}\triangle f(\eta_{i}(g))+\frac{1}{n+1}\triangle f(\eta_{j}(g))
\end{displaymath}
and thus that:
\begin{multline*}
Y_{i}(g+ij)-Y_{i}(g) = q_{i}(g+ij)^{2}-q_{i}(g)^{2}=[q_{i}(g+ij)-q_{i}(g)][q_{i}(g+ij)+q_{i}(g)]\\
                 =q_{i}(g) \left( \frac{2}{n+1} \right) \left[\triangle f(\eta_{j}(g))-n\triangle f(\eta_{i}(g))  \right] +\left( \frac{1}{n+1} \right)^2 \left[\triangle f(\eta_{j}(g))-n\triangle f(\eta_{i}(g)) \right]^2
\end{multline*}
By Condition (5), $\triangle f(\eta_{j}(g))-n\triangle f(\eta_{i}(g))>0$ and hence $Y_{i}(g+ij)-Y_{i}(g) \geq 0$.  Similarly, $Y_{j}(g+ij)-Y_{j}(g) \geq 0$, which implies that both nodes $i$ and $j$ may increase their profit by linking together and so graph $g$ is not stable.  This is a contradiction and so $g^c$ is the only stable graph. This completes the proof.
\end{proof}

\begin{example}
\label{ex: oligopoly complete collaboration}
We present a numerical example of Theorem \ref{thm: negative dv dt complete graph}. Let $N=5$ firms compete in an oligopoly with demand function $P=30-Q$, fixed cost $\gamma_{0}=5$, and $f(\eta_{i}(g))=\frac{1}{\eta_{i}(g)+3}$.  We want to test the stability of the complete graph $g^{c}$. Conditions (1)-(4) are easy to check:
\begin{enumerate*}
\item $\tfrac{d f}{d \eta}=\frac{-1}{(\eta_{i}(g)+3)^{2}}< 0$, which implies $f$ is decreasing;
\item $\tfrac{d^{2} f}{d \eta^{2}}=\frac{2}{(\eta_{i}(g)+3)^{3}}>0$, which implies that $f$ is convex;
\item $f(\eta_{i}(g))=\frac{1}{\eta_{i}(g)+3}>0$, which implies $f$ is positive;
\item $\alpha - \gamma_{0}=30-5=25 > n f(0)=5\left(\frac{1}{3}\right)$
\end{enumerate*}
Condition (5) can be verified with the below table:
\begin{center}
\begin{tabular}{|c|c|c|c|c|c|}
  \hline
$\eta_{i}(g)$ & $\eta_{j}(g)$ & $\triangle f(\eta_{i}(g))$ & $\triangle f(\eta_{j}(g))$ & $n\triangle f(\eta_{i}(g))-\triangle f(\eta_{j}(g))$& $n\triangle f(\eta_{j}(g))-\triangle f(\eta_{i}(g))$\\
\hline
0 & 4 & -0.08 & -0.02 & 0.0059 & 0.4\\
\hline
4 & 0 & -0.02 & -0.08 & 0.3988 & 0.01\\
\hline
\end{tabular}
\end{center}
All conditions of Theorem \ref{thm: negative dv dt complete graph} are met. The stability of the complete graph can be verified exhaustively and was done using Matlab.
\end{example}

\begin{remark} The functions described in Theorem \ref{thm: negative dv dt complete graph} generalizes the result of Goyal and Joshi \cite{goyal2006}, since it is clear that a linear function of the vertex degrees (as assumed in \cite{goyal2006}) will satisfy the given criteria. Furthermore, Theorem \ref{thm: stable assym k distn graph} is a further generalization since we define conditions under which a collaboration graph with an \textit{arbitrary} degree distribution will be stable. This is not possible with Goyal and Joshi's model.
\end{remark}

\section{Conclusions}\label{sec:Conclusion}
In this paper we showed that networks with specific structure properties may form as a result of game theoretic interactions.  Specifically, we showed a simple way of constructing a game whose pairwise stable solutions are those graphs with a given degree distribution. As a particular application of network formation via game theoretic principles, we investigated the formation of collaboration networks in oligopolistic competition.  We extended the model of Goyal and Joshi \cite{goyal2003} with a nonlinear cost function that, under particular conditions, admits stable collaboration graphs with an arbitrary degree sequence as well. One limitation of this approach is that we cannot specify an exact graph structure. The degree distribution specification generates a class of stable graphs rather than a single graph. We do not view this as a problem since it can help explain variation in observed situations with the same parameters.

\subsection{Potential Relation to Network Science}
Special structure in networks has been considered in several recent papers \cite{barabasi1999a, newman2003, dorogovtsev2002, albert2002}, that have cut across various subjects including social networks, information networks, and biological networks as well as physical networks such as power grids and road networks.  A review of networks in these disciplines may be found in various articles (see \cite{newman2003, dorogovtsev2002, albert2002} and the references therein).

The network science literature was largely inspired by the observation of macroscopic structural properties (e.g., small world, power law degree distribution) of networks that occurred in several distinct network types (e.g., social, information, biological networks), which have diverse microscopic properties. The network science literature has largely been devoted to finding the mechanisms by which networks form and/or evolve in order to generate the structural properties that are observed.  The momentum in this direction has largely been driven by the statistical physics community \cite{barabasi1999a,newman2003,dorogovtsev2002}, who argue that the phenomena of complex networks (e.g., power laws) may be explained by laws that reach across all complex networks because they are phenomena that are inherent to the complexity of the networks.

Alternatively, there has been recent interest in the structural properties of networks that have been designed via optimization \cite{Alderson2008, Nagurney2008, ADGW03}.  This perspective is motivated largely by the fact that many networks (in the abstract sense) are models for physical networks that are designed by humans to function with particular objectives (or even designed by  nature to serve an evolutionary purpose). These networks are distinctly different from social networks, which are abstract models that describe interactions between actors (e.g., people talking, writing scientific papers or dating). These networks (e.g., power grids, communication networks) are not designed through central coordination, but arise as a result of the objectives of multiple independent actors.  As a result there are structural properties that often exist in these networks that are not explained by models that do not account for these functioning characteristics \cite{doyle2005}.

The work presented in this paper is motivated by this new work. Our perspective is that some networks are formed as a result of the interacting strategies of multiple players, rather than a universal rule for network formation. This is consistent with certain recent observations on (e.g.) power law networks \cite{SP12}. This occurs in collaboration networks.  The formation of the collaboration network results from the strategic decisions of the players. The network will still have particular structural properties that are explained via game theoretic principles. In this paper, we illustrated that an important property to the Network Science community, namely the degree distribution, can emerge as a result of designed game mechanism. As an explanatory model, we suggest that certain networks with \textit{interesting} degree distributions form as a result of strategic player interactions in which each node degree in embedded (or hidden) in the objective function of the node player. The formulation in this paper is simply an explicit representation. In \cite{LGF11e} we illustrate a class of games that generate specific degree distributions in which the degree is not an explicit part of the objective functions of the players.

\section{Future Directions}
\label{sec: Future Directions}
There are several directions this research could take. In \cite{LGF11e} we begin the investigation of graph formation games with specific link bias and identify a game theoretic mechanism that does not explicitly encode the degree distribution of interest into the players' objective functions and yet is still capable of having graphs with arbitrary degree distributions as stable solutions. This work is extended in \cite{LGF12a}. We also investigate the graph formation game in the presence of spatial oligopolies in \cite{LGF11c}. Clearly investigating the mechanism design problem for additional graph properties, such as the clustering coefficient is of interest. Recent work on generation of graphs with specific clustering coefficients \cite{GRKF11} may provide insight into game theoretic mechanisms for solving the same problem.

Another equally interesting future direction lies in the extraction of game theoretic mechanisms from real-world data. Solving such a problem may take two forms. In the first form, small real world human networks are evaluated using traditional psycho-social interviewing techniques in an attempt to identify objective functions or constraints consistent with the model presented in this paper and theoretical extensions (e.g., \cite{LGF12a}). This work is planned in collaboration with social psychologists at the authors' parent institution. A second form of objective function inference would be completely observational in which, given a network that appeared to be stable, one would attempt to infer the set of potential objective functions (and player constraints \cite{LGF11e}) that would generate the given stable graph. Observation of network evolution would also be useful. In this case, statistical techniques would have to be applied to the determination of the objective function structure.

A final direction of investigation lies in the evaluation of dynamic network formation. That is, investigating this problem in a dynamic context in which the objective functions of players may change. We propose that this problem can be analyzed in discrete time using techniques from competitive Markov decision theory \cite{CMDP}. However, the combinatorial properties of the graph structures (for all but the simplest networks) make a brute force analysis approach intractable. Thus a more sophisticated analysis technique may be necessary to obtain any useful results.

\bibliographystyle{alpha}
\bibliography{Biblio-Database2}

\end{document}